\documentclass{amsart}
\pdfoutput=1
\input cyracc.def

\usepackage{tikz-cd}
\usepackage[vcentermath]{youngtab}
\usepackage{genyoungtabtikz}
\usepackage{amsxtra}
\usepackage{amsmath}%
\usepackage{amsfonts}
\usepackage{amsthm}
\usepackage{amssymb}%
\usepackage{graphicx}
\usepackage[all]{xy}
\usepackage{amscd}
\usepackage[utf8]{inputenc}
\usepackage[english]{babel}
\usepackage{amstext} 
\usepackage{array,tabu}   
\newcolumntype{C}{>{$}p{4cm}<{$}}
\usepackage{hyperref}
\usepackage{booktabs}

\calclayout            

\newcommand{\lr}{\langle}
\newcommand{\rr}{\rangle}
\newcommand{\maxideal}{\mathfrak{m}}
\newcommand{\C}{\mathbb{C}}
\newcommand{\N}{\mathbb{N}}
\newcommand{\Z}{\mathbb{Z}}

\newcommand{\X}{\mathbb{C}^2}
\newcommand{\Hil}{H}
\newcommand{\HilbCn}{H^{[n]}}
\newcommand{\Hilbn}{H^{[n]}}
\newcommand{\Bn}{B^{[n]}}
\newcommand{\Hilnr}{H^{[n,n+r]}}
\newcommand{\Bnnr}{B^{[n,n+r]}}
\newcommand{\Xstar}{Y_0}

\newcommand{\xen}{H^{[n]}}
\newcommand{\aj}{\mathrm{HCh}}
\newcommand{\Grrm}{\mathrm{Gr}_r\left(\C^m\right)}
\newcommand{\sqt}{\sqrt{t}}
\newcommand{\Yng}{\Delta}
\newcommand{\bxnoIn}{\square}

\makeatletter
\def\moverlay{\mathpalette\mov@rlay}
\def\mov@rlay#1#2{\leavevmode\vtop{%
		\baselineskip\z@skip \lineskiplimit-\maxdimen
		\ialign{\hfil$\m@th#1##$\hfil\cr#2\crcr}}}
\newcommand{\charfusion}[3][\mathord]{
	#1{\ifx#1\mathop\vphantom{#2}\fi
		\mathpalette\mov@rlay{#2\cr#3}
	}
	\ifx#1\mathop\expandafter\displaylimits\fi}
\makeatother

\newcommand{\bigcupast}{\charfusion[\mathop]{\bigcup}{\ast}}
\theoremstyle{plain}
\newtheorem{theorem}{Theorem}

\newtheorem{corollary}[theorem]{Corollary}
\newtheorem{definition}[theorem]{Definition}
\newtheorem{example}[theorem]{Example}
\newtheorem{lemma}[theorem]{Lemma}

\newtheorem{proposition}[theorem]{Proposition}
\newtheorem{remark}[theorem]{Remark}

\makeatletter
\DeclareRobustCommand\widecheck[1]{{\mathpalette\@widecheck{#1}}}
\def\@widecheck#1#2{%
	\setbox\z@\hbox{\m@th$#1#2$}%
	\setbox\tw@\hbox{\m@th$#1%
		\widehat{%
			\vrule\@width\z@\@height\ht\z@
			\vrule\@height\z@\@width\wd\z@}$}%
	\dp\tw@-\ht\z@
	\@tempdima\ht\z@ \advance\@tempdima2\ht\tw@ \divide\@tempdima\thr@@
	\setbox\tw@\hbox{%
		\raise\@tempdima\hbox{\scalebox{1}[-1]{\lower\@tempdima\box
				\tw@}}}%
	{\ooalign{\box\tw@ \cr \box\z@}}}
\makeatother

\newcommand\scalemath[2]{\scalebox{#1}{\mbox{\ensuremath{\displaystyle #2}}}}

\definecolor{per}{rgb}{0,0.3,0.6}
\definecolor{per2}{rgb}{0,0.5,0.6}
\definecolor{per3}{rgb}{0.1,0.2,0.5}
\hypersetup{
	bookmarks=true,         
	unicode=false,          
	pdftoolbar=true,        
	pdfmenubar=true,        
	pdffitwindow=false,     
	pdfstartview={FitH},    
	pdftitle={My title},    
	pdfauthor={Author},     
	pdfsubject={Subject},   
	pdfcreator={Creator},   
	pdfproducer={Producer}, 
	pdfkeywords={keyword1, key2, key3}, 
	pdfnewwindow=true,      
	colorlinks=true,       
	linkcolor=per2,          
	citecolor=per3,        
	filecolor=per3,      
	urlcolor=per3           
}

\begin{document}
\title[Refined Hilbert schemes,
  E-polynomials, and the number of generators of
  ideals]{Refined Hilbert schemes,
  E-polynomials, and the number of generators of finite colength
  ideals in the plane}
	\author{Yi-Ning HSIAO}
	\author{Andras Szenes}
	\address{Section de mathématiques, Université de Genève}
	\email{Yi-Ning.Hsiao@unige.ch}
	\address{Section de mathématiques, Université de Genève}
	\email{Andras.Szenes@unige.ch}
	
	\begin{abstract}
          The study of the stratification associated to the number of
          generators of the ideals in the punctual Hilbert scheme of
          points on the affine plain goes back to the '70s.  In this
          paper, we present an elegant formula for the E-polynomials
          of these strata.
	\end{abstract}
	\maketitle
	\setcounter{section}{-1}
	\section{Introduction}
	The study of the topology of the Hilbert scheme of points on the
	affine plane has brought a wealth of results in several branches of
	mathematics, such as  geometric
	representation theory, theory of symmetric polynomials, singularities, symplectic geometry and in the
	enumerative geometry in two dimension
	 (\cite{Cheahcohomology,Go1,haiman, Ia, MVRshortproof, MLlect, HNlecture, NHHeisenberg, homfly}). 
	In this article, we study the strata of
	the Hilbert scheme of points on the complex plane associated to the number of generators of the elements of this space considered as	ideals.
	
	Let $X=\C^2$ be the affine plane and denote by $\xen$ the Hilbert scheme of $n$ points
	on $X$; this is a space parameterizing colength-$n$ ideals
	of the polynomial ring $R=\C[x,y]$:
	\[ \xen=\{ I\trianglelefteq R~|~ \dim_{\C}R/I=n\}.\] 
	The space $\xen$ is naturally endowed with the structure of a
	smooth $2n$-dimension complex variety (\cite{Foga}), and a universal
	sequence of R-modules:
	\[  I\to R\to R/I \]
	over $I\in \Hilbn$.  
	
	Counted with multiplicities, the support of an ideal $I\in\xen$ is $n$
	points in the plane, i.e. an element of $S^n \X$, the $n$th symmetric
	product of $\X$. We will represent the
	elements of $S^n\X$ as formal sums of points in $\X$, and the
	resulting morphism, the Hilbert-Chow morphism, will be denoted by  $ \aj:\xen\to S^n\X $.

	The \emph{Briançon variety} or \emph{the punctual Hilbert scheme}
	$B^{[n]}$ is the subvariety of $\Hilbn$ which collects all ideals
	supported at $(0,0)$:
	\[B^{[n]}=\{ I\trianglelefteq R~|~ \dim_{\C}R/I=n,~
	\aj(I)=n\cdot(0,0)\}.\] The variety $B^{[n]}$ is a deformation retract of $\Hilbn$. It is
	$(n-1)$-dimensional, compact and irreducible. $B^{[n]}$ is, however, not
	smooth: it is non-singular only for $n=1,2$
	(\cite{Brian}).
	
	Given an ideal $I\in \Hilbn$, there are distinct points $p_1,\dots, p_k\in X$ such that $I$ may be expressed as the intersection of ideals \begin{equation}\label{IntersectionFormOfI}
	I=I_{p_1}\cap\dots\cap I_{p_k},
	\end{equation} where $I_{p_i}$ is supported at $\{p_i\}$ for each $i=1,\dots ,k$. 
	We will write $I_0$ for $I_p$ with $p=(0,0)$. 
	Let 
	$$\sigma\left(I\right)=  \dim_\C \left( R/I_0\right)$$
	be  the multiplicity of $I$ at $(0,0)$; with this notation then $I_0$ is an element of $B^{[\sigma(I)]}$.
	
	Denote by $\maxideal:=\lr x, y\rr$  the maximal ideal of $R$ at $(0,0)$, and introduce another $\mathbb{N}$-valued function on $\Hilbn$:
	the minimal number of generators of $I_0$:
	\begin{align*}
	\mu\left( I\right)=\dim_{\C} \left( I_0/\maxideal I_0\right).
	\end{align*}
	Indeed, by Nakayama's Lemma, every basis of $I_0/\maxideal I_0\simeq
	\C^{\mu(I)}$ lifts to a minimum set of generators of $I_0$
	considered as an $R$-module.  This function is a classical
        invariant studied by A. Iarrobino in  \cite{Ia} (See $\S$2).
	
	An important set of examples of ideals in $\Hilbn$ are the ideals
	generated by monomials. In this case, those monomials that are not
	contained in $I$ form a basis 
		\[\left\{(p,q)\in \N^2 \;\vrule\; x^py^q\notin I  \right\}\]
	of the $n$-dimensional space
	$R/I$.
	\begin{figure}[h]
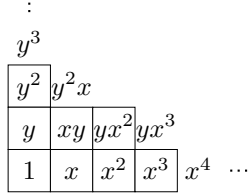

		\vspace*{-0.5cm}
		\Yboxdim{16pt}
		\gyoung(:\vdts,:<y^3>,<y^2>:<{\scalemath{1}{y^2x}}>,y<xy><{ \mbox{$\scalemath{1}{yx^2}$}}>:<{\scalemath{1}{yx^3}}>,1x<x^2><x^3>:<x^4>:\hdts)\vspace*{-0.2cm}
		\caption{The diagram corresponding to the ideal $\lr y^3, y^2x,yx^3,x^3\rr$.}
	\end{figure}
	This defines a one-to-one correspondence
	between monomial ideals in $\Hilbn$ and Young diagrams of $n\in \N$
		, which are in bijection with partitions of $n$:   
	We denote by $\Pi(n)$ 
	the set of all $n$-codimensional monomial ideals of $R$.
	Note that given a diagram representing a monomial ideal $I$, the value
	of $\mu(I)$ is the number of "concave corners" of the diagram:
	\begin{figure}[h]
		$\begin{array}{cc}
		\Yvcentermath1 I:{\Yboxdim{10pt}\gyoung(:\star,;;;:\star)}
		,& \Yvcentermath1 J:{\Yboxdim{10pt}
			\gyoung(:\star,;:\star,;;:\star)}\in H^{[3]}\\
		\lr x^3,y\rr& \lr x^2,xy,y^2\rr
		\end{array}$ (where $\star\Leftrightarrow$ generator) 
		\caption{\mbox{The monomial ideals $I,J$
				have $\mu\left(I\right)=2$ and $\mu\left(J\right)=3$, respectively.}}
	\end{figure}
	
	In the present paper, we study the geometric and topological invariants of strata associated to $\mu$ and $\sigma$.
	
	We will focus on the \emph{E-polynomial} (or \emph{the
		Hodge-Deligne polynomial}) whose  motivic properties  facilitate its
	calculation (\cite{Del1,Del2,HR}).
	Let $Z$ be a complex algebraic variety. P. Deligne established the existence of two filtrations on the $j$-th cohomology group of $Z$: the weight filtration $W_*$
	\[0=W_{-1}\subseteq W_0\subseteq\dots \subseteq W_{2j}=H^{j}(Z)\]
	and the Hodge filtration  $F^*$
	\[H^{j}(Z)=F^0\supseteq F^1\supseteq\dots \supseteq F^m\supseteq F^{2j}=0 \]
	such that, for each $l$, the filtration induced by $F$  on the graded piece $\mathrm{gr}_lW:=W_l/W_{l-1}$
	endows $\mathrm{gr}_lW$ with a pure Hodge structure of weight $l$.
	One can define a mixed Hodge structure on the compactly supported cohomology $H_{c}^*(Z)$ as well (\cite{Fulton}). 
	
	We define the compactly supported mixed Hodge polynomial of $Z$ to be the generating function of the compactly supported mixed Hodge numbers $h^{p,q;j}_c(Z):=\dim_\mathbb{C} \mathrm{gr}_p F \left(\mathrm{gr}_{p+q}W \left(H_c^j(Z)\right)\right)$:
	\[H_c(Z; u,v,s):=\sum_{p,q,j} h^{p,q;j}_c(Z) u^p v^q s^j.\]  
	For a connected smooth variety $Z$ of complex dimension $d$, the cohomological pairing 
	\begin{equation}\label{Eq:PD}
	H^k(Z) \times H_c^{2d-k}(Z) \to H^{2d}(Z).
	\end{equation} is a morphism of mixed Hodge structures. Then Poincaré duality implies the equality of mixed Hodge numbers:
	\begin{equation}\label{Eq:MHnumberscptMHnumbers}
	h^{p,q;j}_c(Z) = h^{d-p,d-q;2d-j}(Z)
	\end{equation} where $h^{p,q;j}:=\dim_\mathbb{C} \mathrm{gr}_p F\left( \mathrm{gr}_{p+q}W (H^{j}(Z))\right)$ are the mixed Hodge numbers of $Z$. Equivalently, if $H(Z; u,v,s):=\sum_{p,q,j} h^{p,q;j}(Z) u^p v^q s^j$ is the mixed Hodge polynomial of $Z$, then we have
	\begin{equation}
	H_c(Z; u,v,s)=u^dv^ds^{2d}H(Z; u^{-1},v^{-1},s^{-1}).
	\end{equation}
	The specialization $H_c(Z; u,v,s)|_{u=1,v=1}$ is the compactly supported Poincaré polynomial of $Z$. 
	Moreover, if $Z$ is a connected smooth variety,
	then by (\ref{Eq:MHnumberscptMHnumbers}), we can express the compactly supported Poincaré polynomial in terms of $P\left(Z;s\right):= \sum_j \dim_\C H^j(Z)s^j$, the Poincaré polynomial of $Z$ as follows:
	\begin{equation}
	H_c(Z; 1,1,s)={P}\left(Z ; s^{-1}\right)s^{2d}.
	\end{equation}
	
	The \emph{E-polynomial} or \emph{the Hodge-Deligne polynomial} of $Z$ is defined to be the compactly supported mixed Hodge polynomial  specialized at $s=-1$
	\[E(Z; u, v):=\sum_{p,g,j} h^{p,q;j}_c(Z)(-1)^j u^p v^q.\] 
	
	{\bf Properties of the E-polynomial:}
	\begin{itemize}
		\item {\bf Additivity:} If a complex algebraic variety  $Z$ is represented as a disjoint union of locally Zariski closed subsets $Z=\cup_{i=1}^n Z_i$, then \[E(Z;u,v)=E(Z_1; u,v)+\dots E(Z_n;u,v) .\]
		\item {\bf Factorization on fibrations:}
		If $f: Z \to B$ is a Zariski locally trivial fibration of complex algebraic varieties with fiber $F$ over a closed
		point, then \[E(Z ) = E(B) \cdot E(F).\]
		In particular, if $B,F$ are complex algebraic varieties, then 
		\[E(B \times F ) = E(B) \cdot E(F ).\]
		\item 
		The specialization at $u=v=1$ gives the topological Euler characteristic 
		\[E(Z; 1,1)=\sum_{p,g,j} (-1)^jh^{p,q;j}_c(Z)=\chi(Z).\]
		When $Z$ is smooth and projective, the specialization
		\begin{align*}
		E\left(Z;-u,-v\right)=\sum_{p,q}\dim H^q(Z,\Omega^p) u^pv^q= \sum_{p,q}h^{p,q} u^pv^q
		\end{align*} agrees with the Hodge polynomial of $Z$, and 
		\begin{align*}
		E(Z; -y,1)=\sum_{p,g,j} (-1)^jh^{p,q;j}(Z)y^p=\chi_y(Z)
		\end{align*} is the Hirzebruch $\chi_y$-genus of $Z$.
	\end{itemize}
	\begin{example}
		The E-polynomial of $\C$ is
                $E\left(\C;u,v\right)=uv$. Applying the factorization
                property, we have $E\left(\C^n;u,v\right)=(uv)^n$, and
                the  $E\left(\C^n\setminus{0};u,v\right)=(uv)^n-1$.
	In particular, the E-polynomial of $\C^*$ is $uv-1$.
	\end{example}
	{\bf Notation}: 
	Throughout the article, $t$ will be a variable of degree 2.
	If the mixed Hodge numbers $h^{p,q;j}(Z)$ of an algebraic variety $Z$
	vanish except when $p=q$, then the E-polynomial $E(Z;u,v)$ may be
	written as a polynomial in $t=uv$.  This condition holds for the
	Hilbert schemes $\Hilbn$, $\Bn$ and the refined Hilbert scheme $\Hilnr$ to be introduced later.  
	Thus we	adopt the simplified notation $E\left(Z;t\right)$ in this article for the
	E-polynomial:
	\begin{equation}
	E\left(Z;t\right) :=E\left(Z;\sqrt{t},\sqrt{t}\right)= \sum_{p,q,j}h_c^{p,q;j}\left(Z\right)u^p v^q s^j \mid_{u,v= \sqrt{t} \text{ and } s=-1 }
	\end{equation}

	The calculations of the E-polynomials of Hilbert schemes of points on smooth
	surfaces goes back to the works of L. G\"{o}ttsche and J. Cheah ( \cite{Cheahcohomology,Cheah,Gobook}). A version of their result
	is the following.
	
	\begin{theorem}[J. Cheah  \cite{Cheahcohomology,Cheah}]
		The generating function of the E-polynomials of a smooth surface $S$ has the form
		\begin{equation}
		\sum_{n=0}^\infty E\left(S^{[n]};u,v\right)s^n=
		\prod_{d=1}^\infty \prod_{p,q} \left(\frac{1}{1-u^{p+d-1}v^{q+d-1}s^d}\right)^{e_{p,q}(S)},
		\end{equation} where $e_{p,g}:=\sum_k (-1)^k h^{p,q,k}_c(Z)$.
	\end{theorem}
	
	In particular, for the case of $\X$, we have
	\begin{equation}\label{Formula:EpolyHn}
		\sum_{n=0}^\infty E\left(\Hilbn;t \right) q^{n}
		=\prod_{d=1}^{\infty} \frac{1}{1-t^{d+1}q^d}.
	\end{equation}
	Grojnowski (\cite{Gr}) and Nakajima (\cite{NHHeisenberg})
        introduced a representation-theoretic structure into the
        picture, which has made a deep impact on the whole subject.
	
	In this article, we present a refinement of these formulas by
	calculating the E-polynomials of the strata  \[\Bn_m:=\left\{I\in \Bn
	\;\vrule\; \mu(I)=m \right\}\]
	associated to the
	invariant $\mu(I)$ of ideals introduced above. 
	\begin{theorem}\label{genEBmn}
		The generating function of the E-polynomials of the strata $\Bn_m$ of the Briançon variety is given by
		\begin{equation}\label{Formula:genEBmnthm}
		\sum_{n=0}^{\infty}	E\left(B^{[n]}_m ;t\right) q^n =
		\prod _{i=1}^{m-1} \frac{1}{1-t^{ i+1}}\cdot \sum _{a=1}^m \left((-1)^{a+1} t^{\binom{a}{2}+m-1} {m \brack a}_{t}  \prod _{k=0}^{\infty} \frac{1-q^k t^{k- a}}{1-q^k t^{ k-1}}\right),
		\end{equation}
		\[\mbox{ where }{m \brack a}_{t}=\displaystyle\begin{cases}\displaystyle
		\prod_{i=0}^{a-1} \frac{1-t^{m-i}}{1-t^{i+1}}  &\text{ if } 1\le a\le m\\
		1	&\text{ if } a=0\\
		0   &\text{ if } a>m.
		\end{cases}
		\] 
	\end{theorem}
	 \begin{remark}
           Note that the summand indexed by $a=1$ on the right hand
           side of \eqref{Formula:genEBmnthm} does not depend on $q$,
           and thus may be omitted when calculating 
           $E\left(B^{[n]}_m ;t\right)$ for $n>0$.
	 \end{remark}
	\begin{example}
		We list some examples of the generating function from Theorem \ref{genEBmn}:
		\begin{itemize}
			\item {\bf Case of $m=2$}:
			\begin{align*}
			\sum_{n=0}^{\infty} E\left(B^{[n]}_2; t\right) q^n
			&=\frac{t}{1-t}+
			\frac{t^2}{t^2-1}\cdot
			\prod_{k=0}^{\infty}
			\frac{ 1- t^{k-2} q^k}{  1- t^{k-1} q^k}\\
			&=\frac{t}{1-t}
			\left(1-
			\prod_{k=1}^{\infty}
		    \frac{  1- t^{k-2} q^k}{  1- t^{k-1} q^k} \right)
			\\
			&=\frac{t}{1-t}
			+
			\prod_{k=1}^{\infty}
			 \left( 1- t^{k-2} q^k\right)
			 \cdot\prod_{k=0}^{\infty}
			 \frac{ 1}{ 1- t^{k-1} q^k}.
			 \end{align*}
			 
			\item {\bf Case of $m=3$}:
			\begin{multline*}
			\sum_{n=0}^{\infty}E\left(B^{[n]}_3; t\right) q^n
			=\frac{t^2}{\left(1-t\right)\left(1-t^2\right) }
			- \frac{t^3}{ \left(1-t\right) \left(1-t^2	\right) }\cdot \prod_{k=0}^{\infty} \frac{ 1-t^{k-2}q^k }{1-t^{k-1}q^k } 
			+ \frac{t^{5}}{\left(1-t^2\right) \left(1-t^3	\right) } \cdot \prod_{k=0}^{\infty} \frac{ 1-t^{k-3}q^k }{1-t^{k-1}q^k }\\
			=\frac{t^2}{\left(1-t^2\right) \left(1-t\right)}
			- \frac{t^2}{ \left(1-t\right)^2 } \cdot \prod_{k=1}^{\infty} \frac{ 1-t^{k-2}q^k }{1-t^{k-1}q^k } 
			+ \frac{t^3}{(1-t^2)(1-t)}\cdot  \prod_{k=1}^{\infty}\frac{ 1-t^{k-3}q^k }{1-t^{k-1}q^k }.
			\end{multline*}
		
		\item {\bf Case of $m=4$}:
		\begin{align*}
		\sum_{n=0}^{\infty} E\left(B^{[n]}_4; t\right) q^n
		=& 
		\frac{t^3}{(1-t) \left(1-t^2\right) \left(1-t^3\right)}
		-\frac{t^4}{(1-t) \left(1-t^2\right)^2}\cdot\prod _{d=0}^{\infty } \frac{ 1- t^{d-2} q^d }{ 1- t^{d-1}q^d }
		\\&+\frac{t^6}{(1-t) \left(1-t^2\right) \left(1-t^3\right)}\cdot\prod _{d=0}^{\infty } \frac{ 1- t^{d-3} q^d }{ 1- t^{d-1}q^d }
		\\&-\frac{t^9}{\left(1-t^2\right) \left(1-t^3\right) \left(1-t^4\right)}\cdot\prod _{d=0}^{\infty } \frac{ 1- t^{d-4} q^d }{ 1- t^{ d-1}q^d }.
		\end{align*}
		\end{itemize}
	\end{example}
	
	A key role in our calculations is played by the \emph{refined incidence varieties} inside $\Hilbn\times H^{[n+r]}$:
	\[\Hilnr:=\left\{(I,J)\in \Hilbn\times H^{[n+r]} \; \vrule\;  I\supset J\supseteq \maxideal I
	\right\}.
	\]
	Note that for $(I,J)\in \Hilnr$, $I/J$ is supported at $(0,0)$.
	
	These spaces were introduced in \cite{perverse} by H. Nakajima and
	K. Yoshioka, where they appeared as examples of moduli spaces of stable
	perverse coherent sheaves on a
	blow-up of $\mathbb{P}^2$ at a point.
	
	The key idea suggested to us by A. Oblomkov is that the fibers of the
	projection $\pi: \Hilnr \to \Hilbn$ over
	$\Hilbn_m=\left\{I \;\vrule\; \mu(I)=m\right\}$ is a Grassmannian:
	\begin{equation}\label{Grstructure}
	\begin{tikzcd}
	\quad\quad\quad\Grrm \rightarrow \pi^{-1} \left(\Hilbn_m\right) \subseteq  \Hilnr\arrow[d, "\pi"]  \\
	\quad\quad\quad\quad\Hilbn_m \subseteq  \Hilbn
	\end{tikzcd}
	\end{equation}
	This idea appears in
	\cite{homfly} and \cite{HOMFLYhomology} in the context of a conjecture relating the HOMFLY
	polynomial of the link of a plane curve singularity $C$ to the
	E-polynomials of the Hilbert schemes of points supported on $C$.
	
    The paper is organized as follows. In Section 1, we introduce the refined Hilbert scheme and state some geometric and topological properties that will be needed later. In Section 2, we prove our main result: Theorem \ref{genEBmn}.
    Then in Section 3, we compute a formula for the generating function of $E\left(\Hilbn_m;t\right)$ as an application of Theorem \ref{genEBmn}.
    In Section 4, we give a formula of the Euler characteristics of $\Bn_m$.
    We list a table of examples of $E\left(\Bn_m;t\right)$ in Appendix A.
	
	{\bf Acknowledgments.} We would like to thank Alexei Oblomkov for key
        suggestions, which started this project. We are also grateful
        to Tamas Hausel, Andrei Negut and Rahul Pandharipande for
        useful discussions. This research was supported by Swiss
        National Science foundation grants 137070, 159581, 156645, 
        175799, and the NCCR SwissMAP.
	
	\section{The refined Hilbert schemes}

	\begin{definition}
		Let $n, r\in \mathbb{N}$, $r\ge1$.
		The refined Hilbert scheme
		$\Hilnr\subset \Hilbn\times \Hilnr$
		is defined as
		\begin{equation}\label{Def:Hilnnr}
		\Hilnr=\{(I,J)\in \Hilbn\times \Hilnr \;\vrule\;  I\supset J\supseteq \maxideal I \}.
		\end{equation}
	\end{definition}
	Here, the condition $I\supset J\supseteq \maxideal I$ is equivalent to the requirement that the 
	quotient space $I/J$ lies in the kernel of the multiplication by $x$ and $y$, in other words, it carries a trivial $R$-module structure.
	It follows that for an element $(I,J)\in \Hilnr$,  the quotient $I/J$ is supported at $(0,0)$.
	
	\begin{example}
		If $(I,J)\in H^{[1,3]}$, then $I/J\simeq \C^2$ is a
                subspace of $I/\maxideal I\simeq \C^{\mu(I)}$. Then
                $\mu(I)$ must be at least 2, and the only ideal $I\in
                H^{[1]}$ of codimension 1 with $\mu(I)\ge 2$ is the
                maximal ideal $\lr x,y\rr=\maxideal$. Moreover, since
                $\dim_{\C}\maxideal/\maxideal^2=2$, the only possible 
                $J\in \Hil^{[3]}$, $J\supseteq \maxideal I$ is $\maxideal^2$. Thus $H^{[1,3]}$ is a point $\left(\maxideal,\maxideal^2\right)$.
	\end{example}
	\begin{theorem}[H. Nakajima, K. Yoshioka, \cite{perverse}]
		The refined Hilbert scheme $\Hilnr$ is smooth and of complex dimension $2n-r(r-1)$.
	\end{theorem}

	The algebraic torus $T\simeq \left(\C^*\right)^2$ acts on
        $\C^2$ by \[t\cdot (x,y)\mapsto (t_1x,t_2y) \]
        \mbox{for $(t_1,t_2)\in T$, $(x,y)\in \C^2$}.  This induces an
        action on the refined Hilbert schemes $\Hilnr$. Furthermore,
        the fixed points $\left(\Hilnr\right)^T$ are parameterized by
        pairs of monomial ideals $(I,J)\in \Hilbn\times \Hil^{[n+r]}$
        such that $I/J\simeq \C^r$ as a trivial $R$-module.  To give a
        description of such $T$-fixed point using Young diagrams, we
        call the boxes of a Young diagram $\Yng$ that have no other
        boxes above and to the right of them, the \emph{elbows} of
        $\Yng$ (cf. Figure \ref{Hinrfixedpt}).
Let $(I,J)\in (\Hilnr)^T$ and $\Yng_I,\Yng_J$ be the corresponding Young diagrams.
    Since $I\supset J$, $\Yng_I$ is a subdiagram of $\Yng_J$. Moreover, the quotient 
	$I/J$ corresponds to a subset $S_{I/J}$ of $\Yng_J$ of $r$ elbows and $\Yng_I=\Yng_J\setminus S_{I/J}$.
	Therefore, we may represent a $T$-fixed point $(I,J)$ by a pair $(\Yng_J, S_{I/J})$ of a Young diagram $\Yng_J$ with $n+r$ boxes, and a subset $S_{I/J}$ of $r$ marked elbows of $\Yng_J$ and we denote the set of these pairs $(\Yng_J, S_{I/J})$ by $\Pi(n,r)$. 
	
	\begin{figure}[h]\label{Hinrfixedpt}
		\vspace*{-0.5cm}
		\[\Yvcentermath1 \Yboxdim{10pt}\left(I=\yng(1,2,3),~ J= \yng(1,2,3,3)~\right)
		~\Leftrightarrow~ \begin{tikzpicture}[scale=1,baseline=-0.5cm]
		\tyoung(0cm,0cm,\star,~\star,~~\star,~~~)[yshift=-2cm]
		\Yfillcolour{yellow}\tgyoung(0cm,0cm,;,:;\star,::;\star).
		\end{tikzpicture} \]
		\caption{\mbox{A fixed point of $H^{[6,9]}$}.}\vspace*{-0.2cm}
	\end{figure}
	
As we will see, the topology of $\Hilnr$ is closely related to the function $\mu$ on $\Hilbn$.
\begin{proposition}\label{Prop:MuAndMaxideal}
	If $I\in \Hilbn$ has $\mu(I)=k$, then $n\ge \frac{k(k-1)}{2}=\binom{k}{2}$. Moreover, the equality holds if and only if $I=\maxideal^{k-1}$.
\end{proposition} 
\begin{proof}  We look for the maximal value of $\mu(I)$ for $I\in \Hilbn$.
	We recall that if $I_\lambda\in(\Hilbn)^T$ and $\Yng_\lambda$ is the corresponding Young diagram with $n$ boxes, then $\mu(I_\lambda)$ is equal to the number of "elbows" of $\Yng_\lambda$.
	
		We claim that 
		\begin{equation}\label{Lemma:MuMax}
		\max\left\{\mu(I) \;\vrule\; I\in \Hilbn\right\}=\max\left\{\mu(I) \;\vrule\; I\in \left(\Hilbn\right)^T\right\}.
		\end{equation}
		To see this, we consider the action of the one-parameter subgroup $\phi:\C^*\to T,~t\mapsto(t,t^N)$ on $\Hilbn$: $t\cdot f(x,y)= f(t^{-1}x, t^{-N}y)$, with  $N\in \N$ large.    
		 For each $I\in \Hilbn$, the limit $t\to 0$ of the action $\phi(t)\cdot(x,y)= (t^{-1}x, t^{-N}y)$ is a $T$-fixed point $I_\lambda$ (\cite{HNlecture}).	
		Let $U_\lambda:=\left\{I\in \Hilbn \;\vrule\;  \underset{t\to 0}{\lim} \phi(t)\cdot I=I_\lambda\right\}$.
		We claim that $\mu\left(I_\lambda\right)\ge\mu\left(I\right)$ for all $I\in U_\lambda$.
		Taking the limit $t\to 0$ of the $\C^*$-action induces the monomial ordering "$y\succ x$" and the limit $\underset{t\to 0}{\lim} \phi(t)\cdot I=I_\lambda$ is, in fact, the initial ideal of $I$ with respect to this ordering $\succ$.   
		Now, if $I\in U_\lambda$ and $G_I$ is a reduced Gröbner basis of $I$, then by definition, $\mu\left(I_\lambda\right)=|G_I|$. Since $I$ is an intersection of ideals with $I_0$ and $\mu(I)=\mu(I_0)$, we have $\mu(I)\le |G_I|= \mu\left(I_\lambda\right)$ by the Buchberger Algorithm construction of the reduced Gröbner basis.
		We conclude that the function $\mu(I), I\in\Hilbn$ reaches its maximum value in $\left(\Hilbn\right)^T$.	

	Now, by equation (\ref{Lemma:MuMax}), the question of finding the maximal value of $\mu(I)$ for $I\in \Hilbn$ reduces to a combinatorial question for Young diagrams.
    In this context, $\max\left\{\mu(I)\;\vrule\;I\in\Hilbn\right\}$ is equal to the maximal number of elbows that a Young diagram with $n$ boxes can have.
	It follows that 
	if $I$ is an ideal with $\dim_\C R/I=\binom{k}{2}$ and $\mu(I)=k$, then it can only be $\maxideal^{k-1}=\lr y^{k-1},y^{k-2}x,\dots,yx^{k-2},x^{k-1}\rr$, which is the monomial ideal corresponding to the partition $\lambda=k-1\ge k-2\ge \cdots\ge 1$.
\end{proof}
    An immediate consequence of Proposition \ref{Prop:MuAndMaxideal} is the following corollary:
    \begin{corollary}
    	We have
    	$\sigma(I)\ge\binom{r}{2}$, and  $\Hilnr$ consists of a single point $\left(\maxideal^{r-1},\maxideal^{r}\right)$ if $n=\binom{r}{2}$.
    \end{corollary}
   \noindent In particular, this shows that 
    the refined Hilbert scheme $\Hilnr$ is empty if $n< \binom{r}{2}$.
    
	According to a theorem of Bialynicki-Birula (\cite{Bi1973,Bi1976}), a $\C^*$-action
	on a smooth projective variety with isolated fixed points induces a
	decomposition of $M$ into affine spaces
	$M =\bigcupast_{p\in M^T}\mathbb{A}^{\alpha(p)}$, where $\alpha(p)$ is
	the number of positive weights of $T_p M$. In this case, the compactly supported Poincaré
	polynomial of $M$ is 
	$\sum_{p\in M^T}
	t^{\alpha(p)}$
	 and it agrees with the $E$-polynomial of $M$.
	 In other words, we have
	 \begin{equation}\label{Eq:EpolyAndCptPoincare}
	 {E}\left(M; t\right)= {P}\left(M; \sqt^{-1}\right)t^{\dim_{\C} M}=\sum_{p\in M^T}
	 t^{\alpha(p)}.
	 \end{equation}
	   Even though the Hilbert schemes of points on the plane are not projective,
	property \ref{Eq:EpolyAndCptPoincare}, nevertheless, holds for $\Hilbn$ endowed with the 
	$\C^*$-action.
	Moreover, this 
	$\C^*$-action on $\Hilbn$ induces, at the same time, a cell decomposition of the Briançon variety $
	\Bn=\bigcupast_{p\in \Pi(n)}\mathbb{A}^{2n-\alpha(p)}$, where $\alpha(p)$ is
	the number of positive weights of $T_p \Hilbn$ (See \cite{ES,HR2,HNlecture}).
	The same arguments go through for
	$\Hilnr$ and $\Bnnr$:
	\begin{proposition}\label{Prop:EqualityEpolyHnnr} We have
		\begin{equation*}
		{E}\left(\Hilnr; t\right)=\sum_{p\in \Pi(n,r)}
		t^{\alpha(p)} \text{~~ and~~~}
		{E}\left(\Bnnr; t\right)=\sum_{p\in \Pi(n,r)}
		t^{2n-r(r-1)-\alpha(p)},
		\end{equation*}
		 where $\alpha(p)$ is
		the number of positive weights of $T_{p} \Hilnr$.
	\end{proposition}
	
	To find the E-polynomial of $\Hilnr$, we apply the character formula for the tangent space of a $T$-fixed point by H. Nakajima and K. Yoshioka in \cite{perverse}. 
	
	\textbf{Notation}: 
	For each two distinct "elbows" $\Yboxdim{8pt}\young(<\star>)$ and $\Yboxdim{8pt}\young(<\blacktriangle>)$ of a Young diagram represented by coordinates $(a,b), (c,d)\in \Z_{\ge 1}^2$ (assuming that $a<c, d<b$), we consider the box $\Yboxdim{8pt}\young(\bullet)=(a,d)$ (cf. Figure \ref{Fig:IntersectionOfElbows}).
		\begin{figure}[h]\vspace*{-0.5cm}\label{Fig:IntersectionOfElbows}
		\[\Yvcentermath1
		\Yboxdim{10pt}\young(\star,~,\bullet\blacktriangle,~~~~) 
		\]
		\centering\text{ An example for $\Yboxdim{8pt}\young(<\star>)=(1,4),\young(<\blacktriangle>)=(2,2)$ and $\Yboxdim{8pt}\young(\bullet)=(1,2)$.}\vspace*{-0.5cm}
	\end{figure}

    Given a pair $\left(\Yng_J, S_{I/J}\right)$, we denote by $Q_{I,J}=\left\{(a,d)\in \Yng_J\;\vrule\;\exists (a,b),(c,d)\in S_{I/J}\text{ with }a<c, d<b \right\}$ the set of boxes obtaining this way.
		
	For a box $\Yboxdim{7pt}\yng(1)$ of a Young diagram, we define ${a(\Yboxdim{7pt}\yng(1))}$ to be the number of the boxes above $\Yboxdim{7pt}\yng(1)$ and $l(\Yboxdim{7pt}\yng(1))$ to be the number of the boxes on right of  $\Yboxdim{7pt}\yng(1)$.
	
	Denote by $T_i$ the one-dimensional representation given by $T_i: (t_1,t_2)\to t_i, i=1,2$.
	\begin{theorem}[H. Nakajima, K. Yoshioka, \cite{perverse}]
		Let $(I,J)$ be a fixed point of $\Hilnr$ of the $T$-action, and $\left(\Yng_J, S_{I/J}\right)$ be the corresponding marked Young diagram. We note that $\Yng_I=\Yng_J\backslash S_{I/J}$ is the Young diagram corresponding to $I$.
		The character of the tangent space $ T_{(I,J)}\Hilnr$ as a $T$-module is given by
		\begin{equation}\label{Eq:HilnrCharacter}
			 \sum_{\bxnoIn }\left( T_1^{-l_\Yng( \bxnoIn )} T_2^{a_{\Yng_I}( \bxnoIn )+1} + T_1^{l_{\Yng _I}( \bxnoIn )+1} T_2^{-a_{\Yng}( \bxnoIn)} \right)
		\end{equation} 
		where the summation runs over all  box $\bxnoIn$ of $\Yng_J\backslash\left(S_{I/J}\bigcup Q_{I,J}
\right)$.
	\end{theorem}      
	\begin{figure}[h]
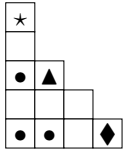
\vspace*{-0.5cm}
		\[\Yboxdim{11pt}\young(\star,~,\bullet\blacktriangle,~~~,\bullet\bullet~\blacklozenge)
		\]\caption{The summation runs over all empty boxes "$\square$".} 
		\label{irrebox} \vspace*{-0.5cm}   	
	\end{figure}		

	\begin{corollary}
		The generating function of the E-polynomial of $\Hilnr$ has the form
		\begin{equation}\label{genEpolyrelHnnr}
		\sum_{n=\binom{r}{2}}^\infty E\left(\Hilnr;t \right) q^{n}
		=q^{\binom{r}{2}} \left(\prod_{d=1}^{\infty} \frac{1}{1-t^{(d+1)}q^d} \right)
		\left(\prod_{d=1}^{r} \frac{1}{1-t^{d} q^d} \right).
		\end{equation}
	\end{corollary}
	
	\begin{proof}
		By Proposition \ref{Prop:EqualityEpolyHnnr}, this is equivalent to find its generating function of the compactly supported Poincaré polynomials of $\Hilnr$.
		We apply the argument of [\cite{perverse} Corollary 5.3 and 5.4] for the Poincaré polynomial of $\Hilnr$.
		To compute $E\left(\Hilnr;t\right)$, we count the sum of weights with opposite sign in their calculation for the compactly supported Poincaré polynomial. 
        Hence we obtain equation (\ref{genEpolyrelHnnr}).
	   	\end{proof}

	\section{Proof of \bf{Theorem} \ref{genEBmn}}
	We consider the refined Hilbert scheme strata
	\begin{align*}
	\HilbCn_m&:=\left\{I\in \HilbCn \;\vrule\; \mu(I)=m\right\},\\
	\HilbCn_m(s)&:=\left\{I\in \HilbCn ~\vrule~ \mu(I)=m, \sigma(I)=s \right\}
	\end{align*}
	and it induces decomposition of $\Hilbn$
	\[ \Hilbn= \bigcupast_m \Hilbn_m.\] 
	
	If $J\in \Hilnr$ satisfies the condition $I\supset J\supseteq\maxideal I$, then $J$ is fully determined by its image in  $I/\maxideal I$.
	Thus for a fixed $I\in \Hilbn$ with $\mu(I)=m$, the set of $J\in H^{[n+r]}$ such that $(I,J)\in \Hilnr$ is parameterized by the Grassmannian of $r$-dimensional subspaces of $I/\maxideal I\simeq \mathbb{C}^{m}$. Over each stratum $\Hilbn_m$ of $\Hilbn$, the projection map $ \Hilnr\to \Hilbn$ has fibers $Gr(r,\mathbb{C}^m)$ at each ideal $I\in \Hilbn_m$.
	
	Since Grassmannians are projective and smooth, their
        $E$-polynomials and Poincaré polynomials are equal:
	\[E\left(Gr(r, \mathbb{C}^m); t\right)
	=P\left(Gr(r,\mathbb{C}^m); \sqt \right)
	={m \brack r}_{t},\] where ${m \brack r}_{t}:=\prod_{i=1}^{r}\frac{1-t^{m-i+1}}{1-t^{i}}$.
	This Grassmannian bundle structure together with the motivic property of the E-polynomial imply the following equality
	\begin{equation}\label{relation1}
	E\left( \Hilnr;t\right) =\sum_{m=1}^{\mu_{n}^{\max}} E\left( \Hilbn_m;t\right)E\left( Gr_r(\C^m);t\right)=\sum_{m=1}^{\mu_{n}^{\max}} E\left( \Hilbn_m;t\right) {m \brack r}_{t},
	\end{equation}
	where $\mu_{n}^{\max}:=\max\left\{\mu(I) \;\vrule\; I\in \Hilbn \right\}=\max\left\{\mu(I) \;\vrule\; I\in \left(\Hilbn\right)^T \right\}$.
	
	Furthermore, presenting an ideal $I\in\Hilbn$ as an intersection of ideals $I_0\cap I'$, where $I_0$ is the part supported on $(0,0)$ and $I'$ s the part supported at $\Xstar:=\C^2 \backslash \{(0,0)\}$, we obtain the decomposition
	\begin{equation}\label{Relation:DecomHnm}
	\Hilbn_m\simeq \bigcupast_{s=0}^{n} \left(B_m^{[s]}\times \Xstar^{[n-s]}\right).
	\end{equation}
	Using the motivic properties of the E-polynomial, we can conclude
	\begin{equation}\label{relation2}
	E \left( \Hilbn_m ;t \right) =\sum_{s=0}^{n} E \left( \Xstar^{[n-s]} ;t \right)\cdot E \left( B_m^{[s]} ;t \right).
	\end{equation}
	
	Our goal is to find the E-polynomials of $ H_m^{[k]}$ and  $B_m^{[k]}$ using equations (\ref{relation1}) and (\ref{relation2}).
	We will consider all $ H_m^{[k]}$ and  $B_m^{[k]}$, $m,k\in \N$, at the same time and we define following infinite matrices:
	
	$\mathcal{X}:=\begin{pmatrix}E\left( H^{[j]}_i; t\right)\end{pmatrix}_{i\ge 1,j\ge 0}$, $ \mathcal{B}:=\begin{pmatrix}E\left( B^{[j]}_i; t\right)\end{pmatrix}_{i\ge 1, j\ge 0}$, 
	$\mathcal{R}:=\begin{pmatrix}E\left( H^{[j,j+i]}; t\right)\end{pmatrix}_{i\ge 1, j\ge 0}$,
	\[\mathcal{G}:= \begin{pmatrix}E\left(Gr_i(\mathbb{C}^j);t\right)\end{pmatrix}_{i,j\ge 1}
	=\begin{pmatrix} {j \brack i}_{t} \end{pmatrix}_{i,j\ge 1} \text{ and }    \mathcal{A}:=\begin{pmatrix}E\left( \Xstar^{[j-i]}; t\right)\end{pmatrix}_{i,j\ge 1}.
	\]
	
	\begin{proposition}
		The matrices $\mathcal{X}, \mathcal{B}, \mathcal{R}, \mathcal{G}$ and $\mathcal{A}$ satisfy
		\[
		\mathcal{G}\mathcal{X}=\mathcal{R} \text{ and }
		\mathcal{B}\mathcal{A}=\mathcal{X}.
		\]
	\end{proposition}
	\begin{proof}
		A direct calculation of matrix products gives
		\begin{align*}
		\mathcal{G}\mathcal{X}&=
		\begin{pmatrix}
		\sum_{k=1}^\infty \mathcal{G}_{ik}\mathcal{X}_{kj}
		\end{pmatrix}_{i\ge 1,j\ge 0}\\
		&=\begin{pmatrix}\sum_{k=1}^\infty E\left(Gr_i(\mathbb{C}^k);t\right)
		E\left( H_k^{[j]}; t\right)\end{pmatrix}_{i\ge 1, j\ge 0}\\
		\left(\mbox{by equation (\ref{relation1}) }\right)&=\begin{pmatrix}
		E\left( H^{[j,j+i]}; t\right)
		\end{pmatrix}_{i\ge 1, j\ge 0}=\mathcal{R}.
		\end{align*}
		Similarly, we have the product $\mathcal{B}\mathcal{A}$
		\begin{align*}
		\mathcal{B}\mathcal{A}&=
		\begin{pmatrix}
		\sum_{k=0}^\infty \mathcal{B}_{ik}\mathcal{A}_{kj}
		\end{pmatrix}_{i\ge 1,j\ge 1}\\
		&=\begin{pmatrix}
		\sum_{k=0}^\infty E\left( B^{[k]}_i; t\right)E\left( \Xstar^{[j-k]}; t\right)
		\end{pmatrix}_{i\ge 1,j\ge 1} \\
		\left(\text{by equation (\ref{relation2}) }\right) &= \begin{pmatrix}
		E\left( H^{[j]}_i; t\right)
		\end{pmatrix}_{i\ge 1, j\ge 0}	=\mathcal{X}.
		\end{align*} 
	\end{proof}

	By definition,  $\mathcal{G}$ and $\mathcal{A}$ are upper triangular matrices with 1s on the diagonal, so they are invertible with upper triangular inverses.
	Then the matrices $\mathcal{X}$ and $\mathcal{B}$ may be expressed as products of matrices
	\begin{equation}\label{matrixrela}
	\begin{cases}
	\mathcal{X}&=\mathcal{G}^{-1}\mathcal{R}\\
	\mathcal{B}&=\mathcal{X}\mathcal{A}^{-1}=\mathcal{G}^{-1}\mathcal{R}\mathcal{A}^{-1}
	\end{cases}
	\end{equation}
	
	Thus to compute the E-polynomials $E\left( \Hilbn_m; t\right)=\mathcal{X}_{m,n}$ and $E\left( B^{[n]}_m; t\right)=\mathcal{B}_{m,n}$, it is sufficient to find the inverse matrices of $\mathcal{G}$ and $\mathcal{A}$.
	
	{
		
		\begin{proposition}
			The matrix $\mathcal{G}=\begin{pmatrix} {j \brack i}_{t} \end{pmatrix}_{i,j\ge 1}$ has the inverse
			\begin{align*}
			\mathcal{G}^{-1}
			&=\begin{pmatrix}  (-1)^{j-i} t^{{j-i\choose 2}}  {j\brack i}_{t} \end{pmatrix}_{i,j\ge 1}.
			\end{align*}	
		\end{proposition}
		
		\begin{proof}
			Let $\mathcal{G}^{-1}$ denote the inverse of $\mathcal{G}$.
			Denoted by $\delta_{ij}$ the Kronecker delta function, the $ij$-th entry of the matrix product $\mathcal{G}^{-1}\mathcal{G}$ that is by definition
			\begin{equation}\label{Eq:Deltaij1}
			(\mathcal{G}^{-1}\mathcal{G})_{ij} = \sum_{k=1}^\infty \mathcal{G}^{-1}_{ik}\mathcal{G}_{kj}=
			\sum_{k=1}^{j} \mathcal{G}^{-1}_{ik}
			{j\brack k}_{t}=\delta_{ij}.
			\end{equation}
			We apply the following orthogonality relation for the q-binomial coefficients (\cite{CL} p.118-p.119): For every $0\le i\le j$ one has
			\begin{align*}\delta_{ij}
			=\sum_{k=i}^j (-1)^{k-i} q^{{k-i\choose 2}} {j\brack k}_q {k\brack i}_q
			\overset{\left(\text{{${k\brack i}_q=0$ if $k<i$}} \right)}{\scalebox{9}[1]{=}}
			&\sum_{k=1}^j (-1)^{k-i} q^{{k-i\choose 2}} {j\brack k}_q {k\brack i}_q,
			\end{align*} and we obtain 
			\begin{equation}\label{Eq:Deltaij2}
			\sum_{k=1}^{j} \mathcal{G}^{-1}_{ik}
			{j\brack k}_{t}=\delta_{ij}=\sum_{k=1}^j (-1)^{k-i} t^{{k-i\choose 2}} {j\brack k}_t {k\brack i}_t.
			\end{equation}
			By comparing the coefficients of ${j\brack k}_t$ in (\ref{Eq:Deltaij2}), we have
			\[
			\mathcal{G}^{-1}_{ik}= (-1)^{k-i} t^{{k-i\choose 2}} {k\brack i}_t.
			\]
		\end{proof}
	}
	Our next step is to calculate the inverse of $\mathcal{A}$. To this end, we first need the generating function of E-polynomials of 
	the Hilbert scheme of points on the punctured plane $\Xstar$.
	
	\begin{proposition}[\cite{Cheahcohomology,Cheah,Gobook}]\label{GenFunction:EC2*}
		The E-polynomial $E \left(\Xstar^{[n]}; t \right)$ of the Hilbert scheme of points on the punctured complex plane $\Xstar$ has the generating function
		\[\sum_{n=0}^\infty E \left(\Xstar^{[n]}; t \right)q^n =\prod _{d=1}^{\infty } \frac{1-t^{ d-1}q^d }{1- t^{ d+1}q^d }.\]
	\end{proposition}
	Here, we give a direct proof using the knowledge of $E\left(\Hilnr;t\right)$ and $E(\Bn;t)$.
	\begin{proof}
		First, we observe that from
	the decomposition in \ref{Relation:DecomHnm}: $\Hilbn\simeq \bigcupast_{s=0}^n \Xstar^{[s]}\times B^{[n-s]},$
	 we have a bijective morphism 
		\[\bigcupast_{s=0}^n \Xstar^{[s]}\times B^{[n-s]} \to \Hilbn\] by sending a pair of subschemes in $\Xstar^{[s]}\times B^{[n-s]}$ to the union of the two. 
		
    	We recall that the formula (\ref{Formula:EpolyHn}) for the generating function of the E-polynomials of $\Hilbn$ has the form \begin{equation}
			\sum_{n=0}^\infty E\left(\Hilbn ; t \right) q^n= \prod_{d=1}^{\infty} \frac{1}{1-t^{d+1}q^d}.
		\end{equation}
		Applying the motivicity of the E-polynomial to this decomposition \ref{Relation:DecomHnm}, we obtain the equality
		\begin{equation}\label{EpolyHnmotivic}
		\sum_{n=0}^\infty E\left(\Hilbn ;t \right) q^n= \sum_{n=0}^{\infty} \left( \sum_{s=0}^n E\left( \Xstar^{[s]}; t\right) E\left( B^{[n-s]};t \right) \right)q^n .
		\end{equation}
		After the change of variable $k=n-s$, the right-hand side of the equation ($(\ref{EpolyHnmotivic})$) has the form of a doubly infinite summation 
		\begin{align}\label{eq:EBntoPBn}
		 \sum_{n=0}^{\infty} \left( \sum_{s=0}^n E\left( \Xstar^{[s]}; t\right) E\left( B^{[n-s]};t \right) \right)q^n &\overset{\left(k=n-s \right)}{\scalebox{5}[1]{=}}
			\sum_{k=0}^{\infty} \left( \sum_{s=0}^\infty E\left( \Xstar^{[s]}; t\right) E\left( B^{[k]};t \right) \right)q^{s+k}\nonumber\\
		&=\sum_{k=0}^{\infty} \sum_{s=0}^\infty E\left( \Xstar^{[s]}; t\right)q^s E\left( B^{[k]};t \right) q^{k}\nonumber\\
		&\overset{\left(\text{$s$, $k$ are independent}\right)}{\scalebox{11}[1]{=}}
		\left(\sum_{k=0}^{\infty} E\left( B^{[k]};t \right) q^{k}\right) \left(\sum_{s=0}^\infty E\left( \Xstar^{[s]}; t\right)q^s\right)
		.\end{align}
		As we pointed out in the beginning of the section, the Briançon variety $B^{[k]}$ admits a cell decomposition as well and thus $E\left( B^{[k]};t \right)={P}\left( B^{[k]};\sqt^{-1} \right)t^{\dim_{\C}B^{[k]}}$.
		We replace pieces $E\left( B^{[k]};t \right)$ by ${P}\left( B^{[k]};\sqt^{-1} \right)t^{\dim_{\C}B^{[k]}}$ in equation (\ref{eq:EBntoPBn}):
		\begin{align*}
		&\left(\sum_{k=0}^{\infty} {P}\left( B^{[k]};\sqt^{-1} \right)t^{\dim_{\C}B^{[k]}} q^{k}\right)  \left(\sum_{s=0}^\infty E\left( \Xstar^{[s]}; t\right)q^s\right), \end{align*}
		which is equal to
		\begin{equation}\left(\sum_{k=0}^{\infty} P\left( H^{[k]}; \sqt \right) q^{k}\right)  \left(\sum_{s=0}^\infty E\left( \Xstar^{[s]}; t\right)q^s\right), \end{equation} since  $\Hilbn$ and $\Bn$ are homotopic.
		Recall that $P\left(\Hilbn; \sqt\right)$ has the generating function
		\begin{equation*}
		\sum_{n=0}^\infty P\left(\Hilbn; \sqt\right)q^n=\prod_{d=1}^{\infty} \frac{1}{1-t^{d-1}q^d}.
		\end{equation*} Thus we have
		\[\sum_{n=0}^\infty E\left(\Hilbn ;t \right) q^n=\prod_{d=1}^{\infty} \frac{1}{1-t^{d+1}q^d}=\left( \prod_{d=1}^{\infty} \frac{1}{1-t^{d-1}q^d} \right) \left(\sum_{s=0}^\infty E\left( \Xstar^{[s]}; t\right)q^s\right).
		\]
		Therefore
		\[\sum_{n=0}^\infty E \left(\Xstar^{[n]}; t \right)q^n =\prod _{d=1}^{\infty } \frac{1-t^{ d-1}q^d }{1- t^{ d+1}q^d }.\] 
	\end{proof}
	
	\begin{example} We list some E-polynomials
		$E \left(\Xstar^{[n]}; t \right)$ for $0\le n\le 8$: 
		\[\begin{array}{|c|c|}\hline
		n& E \left(\Xstar^{[n]}; t \right)\\\hline
		0 & 1 \\\hline
		1 & t^2-1 \\\hline
		2 & t^4+t^3-t^2-t \\\hline
		3 & t^6+t^5-2 t^3-t^2+t \\\hline
		4 & t^8+t^7+t^6-t^5-3 t^4+t^2 \\\hline
		5 & t^{10}+t^9+t^8-3 t^6-3 t^5+t^4+2 t^3 \\\hline
		6 & t^{12}+t^{11}+t^{10}+t^9-t^8-4 t^7-3 t^6+3 t^5+2 t^4-t^3 \\\hline
		7 & t^{14}+t^{13}+t^{12}+t^{11}-3 t^9-6 t^8-t^7+5 t^6+2 t^5-t^4 \\\hline
		8 & t^{16}+t^{15}+t^{14}+t^{13}+t^{12}-t^{11}-5 t^{10}-6 t^9+t^8+7 t^7+t^6-2 t^5 \\\hline
		\end{array}\]
	\end{example}
	
	Before stating the result about the matrix $\mathcal{A}^{-1}$, we define the \emph{dual E-polynomial} $\widecheck{E}$ of a complex variety $Z$
	\begin{equation}\label{def:EDual}
	\widecheck{E}(Z;t):= H(Z;\sqt,\sqt,-1).
	\end{equation}

	When $Z$ is a connected and smooth, we can compute $\widecheck{E}(Z;t)$ from ${E}(Z;t)$ by Poincaré duality (\ref{Eq:PD}):
	\begin{align*}
	\widecheck{E}(Z;t)=&E\left(Z;t^{-1}\right)t^{\dim_{\C} Z}.
	\end{align*}

	\begin{corollary}\label{Cor:DualEpolyXstar}
		The  dual E-polynomial of $\Xstar^{[n]}$ has generating function
		\[\sum_{n=0}^{\infty} \widecheck{E} \left(\Xstar^{[n]}; t \right)q^n =\prod _{d=1}^{\infty } \frac{1-t^{d+1}q^d }{1- t^{d-1}q^d }.\]
	\end{corollary}
	\begin{proof}
		Since the  Hilbert scheme $\Xstar^{[n]}$
		is a complex $4n$-dimensional smooth variety, the Poincaré duality  is compatible with the mixed Hodge structure on cohomology, and we have
		\[\widecheck{E} \left(\Xstar^{[n]}; t \right)= t^{\dim_\C \Xstar^{[n]}} E\left(\Xstar^{[n]}; t^{-1} \right)=t^{2n}E\left(\Xstar^{[n]}; t^{-1} \right). \]
		Then from the definition of the generating function of $\widecheck{E}$, we have 
		\begin{align*}
		\sum_{n=0}^{\infty} \widecheck{E} \left(\Xstar^{[n]}; t \right)q^n &=
		\sum_{n=0}^{\infty} E \left(\Xstar^{[n]}; t^{-1} \right) t^{2n}q^n 
		\overset{\left(\tilde{ q}=t^2q\right)}{\scalebox{4.5}[1]{=}}
		\prod _{d=1}^{\infty } \frac{1-t^{1 -d} \tilde q^d }{1- t^{- d-1}\tilde q^d }\\
		&= \prod _{d=1}^{\infty } \frac{1-t^{1-d} (t^2q)^d }{1- t^{-d-1} (t^2q)^d }
		= \prod _{d=1}^{\infty } \frac{1-t^{d+1} q^d }{1- t^{d-1} q^d }.
		\end{align*} 
	\end{proof}

	
	We are now ready to calculate $\mathcal{A}^{-1}$.
	
	\begin{proposition}
		The matrix $\mathcal{A}$ has the inverse
		\[\mathcal{A}^{-1}=\begin{pmatrix}\mathcal{A}^{-1}_{ij}\end{pmatrix}=\begin{pmatrix}\widecheck{E}\left(\Xstar^{[j-i]}\right)\end{pmatrix}\]
		\begin{equation}\label{Ainverse}
		\mathcal{A}^{-1}=\begin{pmatrix}
		1 &  \widecheck{E}\left(\Xstar^{[1]}\right)  &  \widecheck{E}\left(\Xstar^{[2]}\right) & \widecheck{E}\left(\Xstar^{[3]}\right) && \widecheck{E}\left(\Xstar^{[k]}\right) \\
		0& 1 &  \widecheck{E}\left(\Xstar^{[1]}\right)  &   \widecheck{E}\left(\Xstar^{[2]}\right) &  \dots  & \widecheck{E}\left(\Xstar^{[k-1]}\right)&\dots \\
		\vdots &0 & 1 &   \widecheck{E}\left(\Xstar^{[1]}\right) &  &\vdots\\
		&&  0&  1& &\widecheck{E}\left(\Xstar^{[k-j]}\right)\\
		&&  \vdots&  0& &\vdots\\
		&&  &  & &1\\
		\end{pmatrix},
		\end{equation}
		where $\widecheck{E}\left(\Xstar^{[j-i]}\right)= t^{2(j-i)} E\left(\Xstar^{[j-i]};t^{-1} \right)$ is the dual  E-polynomial of $\Xstar^{[j-i]}$.
	\end{proposition}
	\begin{proof}
		Denote by $\mathcal{C}$ be the infinite matrix in (\ref{Ainverse}).
		The $(i,j)$-th entry of the matrix $\mathcal{A}\mathcal{C}$ is given by the sum
		\begin{align*}\left(\mathcal{A}\mathcal{C}\right)_{ij}
		&=\sum_{k=i}^{j} E \left(\Xstar^{[k-i]}; t \right)
		\widecheck{E} \left(\Xstar^{[j-k]}; t \right)\\ 
		&=\sum_{l=0}^{j-i} E \left(\Xstar^{[l]}; t \right)
		\widecheck{E} \left(\Xstar^{[j-i-l]}; t \right).
		\end{align*}
		
		Note that Proposition \ref{GenFunction:EC2*} and Corollary \ref{Cor:DualEpolyXstar} yield the product of the generating functions 
		\begin{align*}
		\sum_{n=0}^{\infty}\left( \sum_{i=0}^{n}
		E \left(\Xstar^{[i]}; t \right) \widecheck{E}\left(\Xstar^{[n-i]}; t \right)
		\right) q^n=\left(\sum_{n=0}^{\infty} E \left(\Xstar^{[n]}; t \right)q^n\right) 
		\left(\sum_{n=0}^{\infty} \widecheck{E} \left(\Xstar^{[n]}; t \right)q^n\right)=1
		.
		\end{align*} 
		Therefore, $\left(\mathcal{A}\mathcal{C}\right)_{ij}=\delta_{ij}$ and this implies $\mathcal{C}=\mathcal{A}^{-1}$.

	\end{proof}
	
	We are now ready to prove Theorem \ref{genEBmn}.
	\begin{proof}[Proof of Theorem \ref{genEBmn}]
		The E-polynomial $E\left( B_m^{[n]};t\right)$ is equal to the $(m,n)$-th entry of the matrix product
		$\mathcal{B}=\mathcal{X}\mathcal{A}^{-1}=\mathcal{G}^{-1}\mathcal{R}\mathcal{A}^{-1}$ which is given by the sum
		\begin{align*}
		\sum_{k}\sum_{j}\mathcal{G}^{-1}_{mk}\mathcal{R}_{kj}\mathcal{A}^{-1}_{jn}=\sum_{k=m}^{\mu_{n}^{\max}} \sum_{j=\frac{k(k-1)}{2}}^{n}
		(-1)^{k-m}  t^{{k-m\choose 2}}   
		{k \brack m}_{t}
		E\left(H^{[j,j+k]}; t\right) \widecheck{E}\left(\Xstar^{[n-j]} ;t \right).
		\end{align*}
		Substituting this into the generating function, we obtain 
		\begin{align*}
		\sum_{n=0}^{\infty}E\left( B_m^{[n]};t\right) q^n &=
		\sum_{n=0}^{\infty} \left( \sum_{k=m}^{\mu_{n}^{\max}} \sum_{j=\frac{k(k-1)}{2}}^{n}
		(-1)^{k-m}  t^{{k-m\choose 2}}   
		{k \brack m}_{t}
		E\left(\Hil^{[j,j+k]}; t\right) \widecheck{E}\left(\Xstar^{[n-j]} ;t \right) \right)q^n.\end{align*}
		Here we can let the indices $k$ and $j$ run from 0 to $\infty$ without changing the infinite sum since the space $\Hil^{[j,j+k]}=\emptyset$ if $j\le \frac{k(k-1)}{2}$ by Proposition \ref{Prop:MuAndMaxideal}, and ${k \brack m}_{t}=0$ if $k\le m$. 
		Recall that the generating functions of $E\left(\Hilnr; t\right)$ and $ \widecheck{E}\left(\Xstar^{[n]} ;t \right)$ are given by 
		 $q^{\binom{r}{2}}\left(\prod_{d=1}^{\infty} \frac{1}{1-t^{(d+1)}q^d} \right)
		\left(\prod_{d=1}^{r} \frac{1}{1-t^{d} q^d} \right)$
		and
		$\prod _{d=1}^{\infty } \frac{1-t^{d+1}q^d }{1- t^{d-1}q^d }$, respectively.
		Then the generating function
		\begin{align*}
		\sum_{n=0}^{\infty}E\left( B_m^{[n]};t\right) q^n 
		&=
		\sum_{n=0}^{\infty} \left( \sum_{k=0}^{\infty} \sum_{j=0}^{\infty}
		(-1)^{k-m}  t^{{k-m\choose 2}}   
		{k \brack m}_{t}
		E\left(\Hil^{[j,j+k]}; t\right) \widecheck{E}\left(\Xstar^{[n-j]} ;t \right) \right)q^n\end{align*}
		is equal to the product  
		\begin{align}\label{Formula:EpolyforChiBnm}
		\prod_{d=1}^{\infty} \frac{1}{1-t^{d+1}q^d}\cdot&
	    \prod _{d=1}^{\infty } \frac{1-t^{d+1}q^d }{1- t^{d-1}q^d } \cdot
		\sum_{k=0}^{\infty}\left( q^{\binom{k}{2}} (-1)^{k-m}  t^{{k-m\choose 2}}   
		{k \brack m}_{t}
		\prod_{d=1}^{k} \frac{1}{1-t^{d} q^d} \right)\nonumber
		\\
		&= 
		 \prod _{d=1}^{\infty } \frac{1}{1- t^{d-1}q^d }\cdot
		\sum_{k=0}^{\infty}\left( (tq)^{\binom{k}{2}} (-1)^{k-m}  t^{-km+\binom{m+1}{2}}   
		{k \brack m}_{t} \frac{1}{(tq)_k}\right),
		\end{align} where $\displaystyle(tq)_k:=\prod_{d=1}^{k} \left(1-t^dq^d\right)$.
		To continue the proof, we need the following lemma.
		{ 
			\begin{lemma}\label{KClem}
				We have
				\begin{equation}\label{Eq:Lemma}
				\sum_{i=0}^{m}  (-1)^{m+i} t^{km-\binom{m}{2}+\binom{i}{2}-ik}  {m\brack i}_t
				= \prod_{i=0}^{m-1}\left(1-t^{k-i}\right).
			\end{equation}	
			\end{lemma}
			\begin{proof}[Proof of Lemma \ref{KClem}]
			Recall the Gauss's binomial formula (\cite{KC}, p.29):
			\begin{equation*}
			\displaystyle \prod_{{k=0}}^{{n-1}}(1+aq^{k})=\sum _{{k=0}}^{n}q^{\binom{k}{2}}{n \brack k}_{q}a^{k}.
			\end{equation*} 
			Applying the Gauss's binomial formula with $a=-t^{-k}, q=t$, we obtain
			\begin{align*}
			\prod_{i=0}^{m-1}\left(1-t^{k-i}\right)&=
			(-1)^mt^{mk-\binom{m}{2}}
			\prod_{i=0}^{m-1}\left(1+(-t^{-k})t^{i}\right)
			\\&=(-1)^mt^{mk-\binom{m}{2}}\sum_{i=0}^{m}t^{\binom{i}{2}}{m \brack i}_{t}(-t^{-k})^i
			\\&=\sum_{i=0}^{m}(-1)^{m+i}t^{mk-\binom{m}{2}+\binom{i}{2}-ik}{m \brack i}_{t}.
			\end{align*}
			\end{proof}
		}
		We continue the calculation of the generating function of $E\left( B_m^{[n]};t\right)$. 
		We write $\displaystyle{k \brack m}_{t}=\prod_{i=0}^{m-1} \frac{1-t^{k-i}}{1-t^{i+1}}$ and apply Lemma \ref{KClem} to the product $\prod_{i=0}^{m-1}\left(1-t^{k-i}\right)$.
		We substitute it into the generating function
		\begin{align}
		\sum_{n=0}^{\infty}E\left( B_m^{[n]};t\right) q^n 
		&=
	   \prod _{d=1}^{\infty } \frac{1}{1- t^{d-1}q^d } \cdot
		\sum_{k=0}^{\infty} \left((tq)^{\binom{k}{2}} (-1)^{k-m}  t^{-km+\binom{m+1}{2}}   
		{k \brack m}_{t}
		\frac{1}{(tq)_k}\right)\nonumber\\
		\begin{split}
		=&\prod _{d=1}^{\infty } \frac{1}{1- t^{d-1}q^d }\\&\times(-1)^{m}
		\sum_{k=0}^{\infty}  \frac{\displaystyle \frac{ (-1)^{k} (tq)^{\binom{k}{2}} }{(tq)_k} 
			t^{-km+\binom{m+1}{2}}  \sum_{a=0}^{m}  (-1)^{m+a} t^{km-\binom{m}{2}+\binom{a}{2}-ak}  {m\brack a}_t }{\displaystyle\prod_{i=0}^{m-1} \left(1- t^{i+1}\right)  } \nonumber
		\end{split}\\
		&=\prod _{d=1}^{\infty } \frac{1}{1- t^{d-1}q^d } \cdot \prod_{i=0}^{m-1} \frac{1}{1- t^{i+1}  } \cdot \sum_{a=0}^{m} \left( (-1)^{a} t^{m+\binom{a}{2}}  {m\brack a}_t
		\sum_{k=0}^{\infty} \frac{ (-1)^{k} (tq)^{\binom{k}{2}}  t^{-ak}}{(tq)_k}\right).\nonumber
		\end{align}
		Recall the Euler identity in (\cite{CL}):
		\begin{equation}\label{Formula:EulerIdentity}
		(z)_{\infty} = \sum_{n=0}^{\infty} \frac{\left(-1\right)^n z^n q^{\binom{n}{2}}}{\left(q\right)_n}=\prod_{n=0}^{\infty}\left( 1- z q^n \right).
		\end{equation}
		We apply the identity (\ref{Formula:EulerIdentity}) to the infinite sum $\displaystyle\sum_{k=0}^{\infty} \frac{ (-1)^{k} (tq)^{\binom{k}{2}}  t^{-ak}}{(tq)_k}$ with change of variables $\begin{cases}
		q\mapsto tq,\\ z\mapsto t^{-a}
		\end{cases}$.
		
		Finally, we arrive at the result:
		\begin{align*}
		\sum_{n=0}^{\infty}E\left( B^{[n]}_m; t \right)q^n 
		&= \prod _{d=1}^{\infty } \frac{1}{1- t^{d-1}q^d } \cdot \prod_{i=0}^{m-1} \frac{1}{1- t^{i+1}  }\cdot  \sum_{a=0}^{m}\left( (-1)^{a} t^{m+\binom{a}{2}}  {m\brack a}_t
		\prod_{k=0}^{\infty} (1-t^{-a} (tq)^k)\right)\\
		&= \prod_{i=0}^{m-1} \frac{1}{1- t^{i+1}  } \cdot\sum_{a=0}^{m}\left( (-1)^{a} t^{m+\binom{a}{2}}  {m\brack a}_t \left( 1- t^{0-1}\right)
		 \prod _{d=0}^{\infty } \frac{1-t^{d-a} q^d}{1- t^{d-1}q^d } \right)\\
		&= \prod_{i=1}^{m-1} \frac{1}{1- t^{i+1}  } \cdot\sum_{a=0}^{m} \left( (-1)^{a+1} t^{m-1+\binom{a}{2}}  {m\brack a}_t 
		 \prod _{d=0}^{\infty } \frac{1-t^{d-a} q^d}{1- t^{d-1}q^d } \right).
		\end{align*}
		Note that the infinite product $\displaystyle\prod _{d=0}^{\infty } \frac{1-t^{d-a} q^d}{1- t^{d-1}q^d } =0$ when $a=0$, and we have 
		\begin{equation*}
		\sum_{n=0}^{\infty}E\left( B^{[n]}_m; t \right)q^n =
		 \prod_{i=1}^{m-1} \frac{1}{1- t^{i+1}  } \cdot \sum_{a=1}^{m} \left((-1)^{a+1} t^{m-1+\binom{a}{2}}  {m\brack a}_t 
		 \prod _{d=0}^{\infty } \frac{1-t^{d-a} q^d}{1- t^{d-1}q^d } \right).
		\end{equation*} 
	\end{proof}

	\section{The E-polynomial of the refined strata $\Hilbn_m$}
	Using Theorem $\ref{genEBmn}$, we calculate the E-polynomial of the refined strata $\Hilbn_m$ as well.
	\begin{proposition}\label{genEXmn}
		The E-polynomial of the refined stratum $\Hilbn_m$ has the generating function
		\begin{align}
		\sum_{n=0}^{\infty}E\left(\Hilbn_m ;t\right) q^n&=
		\prod _{i=1}^{m-1} \frac{1}{1-t^{i+1}}\cdot
		\sum _{a=1}^m\left( (-1)^a t^{\binom{a}{2}+m} {m\brack a}_{t}  \prod _{k=0}^{\infty} \frac{1-q^k t^{k-a}}{1-q^k t^{k+1}}\right).
		\end{align}
	\end{proposition}
	\begin{remark}
		One can also obtain this formula directly from the entries of matrix products $\mathcal{X}=\mathcal{G}^{-1}\mathcal{R}$ in equations $(\ref{matrixrela})$ with a similar argument as in the proof of Theorem \ref{genEBmn}.
	\end{remark} 
	\begin{proof}
		From relation (\ref{relation2}), the E-polynomial of $\Hilbn_m$ is a sum
		\[
		E \left( \Hilbn_m ;t \right) =\sum_{s=0}^{n} E \left(( \Xstar)^{[n-s]} ;t \right) E \left( B_m^{[s]} ;t \right).\]
		Then the generating function of the E-polynomial $E \left( \Hilbn_m ;t \right)$
		\begin{align*}
		\sum_{n=0}^{\infty}E\left(\Hilbn_m; t\right) q^n=&
		\sum_{n=0}^{\infty}\left(\sum_{j=0}^{n} E \left(( \Xstar)^{[n-j]} ;t \right) E \left( B_m^{[j]} ;t \right) \right) q^n\\
		=&
		\left(	\sum_{n=0}^{\infty}E\left(\Xstar^{[n]}; t\right) q^n\right)\left(	\sum_{n=0}^{\infty}E\left(B^{[n]}_m; t\right) q^n\right).\end{align*}
		
		Applying the formula of the generating functions in $(\ref{GenFunction:EC2*})$ and $(\ref{genEBmn})$, we obtain
		\begin{align*}
		\sum_{n=0}^{\infty}E\left(\Hilbn_m; t\right) q^n=&
		\prod _{d=1}^{\infty } \frac{1-t^{ d-1}q^d }{1- t^{ d+1}q^d }\cdot\prod_{d=1}^{m-1} \frac{1}{1- t^{i+1}  }\cdot\sum_{a=1}^{m}\left( (-1)^{a+1} t^{m-1+\binom{a}{2}}  {m\brack a}_t \cdot
		 \prod _{d=0}^{\infty } \frac{1-t^{d-a} q^d}{1- t^{d-1}q^d } \right)\\
		=&\prod_{d=1}^{m-1} \frac{1}{1- t^{i+1}  }\cdot \sum_{a=1}^{m}\left( (-1)^{a+1} t^{m+\binom{a}{2}}  {m\brack a}_t 
		\cdot \frac{t^{-1}(1-t)}{1-t^{-1}}\cdot\prod _{d=0}^{\infty } \frac{1-t^{d-a} q^d}{1- t^{d+1}q^d } \right)\\
		=&\prod_{d=1}^{m-1} \frac{1}{1- t^{i+1}  }\cdot \sum_{a=1}^{m}\left(  (-1)^{a} t^{m+\binom{a}{2}}  {m\brack a}_t 
		\prod _{d=0}^{\infty } \frac{1-t^{d-a} q^d}{1- t^{d+1}q^d } \right).
		\end{align*}
		
	\end{proof}
	
	\begin{example}
		We list some examples of $E\left( \Hilbn_m;t \right)$:
		\begin{equation}\label{table1}	
		\begin{array}{|c|c|c| c|}
		\hline
		n & m=1 & m=2 & m=3 \\\hline
		1 & t^2-1 & 1 & 0 \\ \cline{1-1}
		2 & t^4+t^3-t^2-t & t^2+t & 0 \\\cline{1-1}
		3 & t^{6}+t^{5}-2 t^3-t^2+t & t^4+2 t^3+t^2-t-1 & 1 \\\cline{1-1}
		4 & t^{8}+t^{7}+t^{6}-t^{5}-3 t^4+t^2 & t^{6}+2 t^{5}+3 t^4-2 t^2-t & t^2+t \\\hline
		\end{array}\end{equation}	
	\end{example}
    \vspace*{0.5cm}
	
\section{Euler characteristics of the refined strata} 
		The specialization of $E\left(\Bn_m;t\right)$ at $t=1$ is the topological Euler characteristic of $\Bn_m$ \[\chi(\Bn_m)=E\left(\Bn_m;1\right),\]
		which equals to $\chi\left(\HilbCn_m\right)$.
		It is, however, difficult to obtain an explicit formula for the generating function of $\chi\left(\Bn_m\right)$ from formula (\ref{Formula:genEBmnthm}) in Theorem \ref{genEBmn}.
	    If we consider the specialization of the equation (\ref{Formula:EpolyforChiBnm}) in the proof of Theorem \ref{genEBmn}
	    \begin{align*}\sum_{n=0}^{\infty} E\left(\Bn_m;t\right)q^n= 
	     \prod _{d=1}^{\infty } \frac{1}{1- t^{d-1}q^d } \cdot
	    \sum_{k=0}^{\infty} \left((tq)^{\binom{k}{2}} (-1)^{k-m}  t^{-km+\binom{m+1}{2}}   
	    {k \brack m}_{t} \frac{1}{(tq)_k}\right)
        \end{align*}at $t=1$, we obtain
	     a formula for the generating function of  $\chi(\Bn_m)$:
	    \begin{equation}
	    \sum_{n=0}^{\infty}\chi\left( B_m^{[n]}\right) q^n
	    =\prod _{d=1}^{\infty } \frac{1}{1- q^d } \cdot
	    \sum_{k=0}^{\infty} \left(\frac{(-1)^{k-m}q^{\binom{k}{2}}}{(q)_k} 
	    \binom{k}{m}\right).
	    \end{equation} 
	    \begin{table}[h]
	    	\[\begin{array}{|c|cccc|}\hline
	    	n & m=2& m=3&m=4&m=5\\\hline
	    	0 & 0 & 0 & 0 & 0 \\
	    	1 & 1 & 0 & 0 & 0 \\
	    	2 & 2 & 0 & 0 & 0 \\
	    	3 & 2 & 1 & 0 & 0 \\
	    	4 & 3 & 2 & 0 & 0 \\
	    	5 & 2 & 5 & 0 & 0 \\
	    	6 & 4 & 6 & 1 & 0 \\
	    	7 & 2 & 11 & 2 & 0 \\\hline
	    	\end{array}\]
	    	\caption{Examples of $\chi\left( \Bn_m\right)$.}
	    \end{table}\vspace*{-0.5cm}
	    
	\newpage
	\appendix
	\section{Tables for E-polynomials of the refined strata}
	\begin{table}[h]\caption{E-polynomials of the refined strata of the punctual Hilbert schemes $E\left( B_m^{[n]};t \right)$.}
		\begin{tabular}{|c|>{$}p{3cm}<{$}|>{$}p{3cm}<{$}|>{$}p{3cm}<{$}|>{$}p{3cm}<{$}|}\hline
			$n $  & E\left(B_2^{[n]};t\right) & E\left(B_3^{[n]};t\right) &   E\left(B_4^{[n]};t\right) &  E\left(B_5^{[n]};t\right) \\\hline
			1	 & 1 & 0 & 0 & 0 \\\hline
			2	 & t+1 & 0 & 0 & 0 \\\hline
			3	 & t^2+t & 1 & 0 & 0 \\\hline
			4	 & t^3+2 t^2 & t+1 & 0 & 0 \\\hline
			5	 & t^4+2 t^3-t & 2 t^2+2 t+1 & 0 & 0 \\\hline
			6	 & t^5+3 t^4+t^3-t^2 & 2 t^3+3 t^2+t & 1 & 0 \\\hline
			7	 & t^6+3 t^5+t^4-2 t^3-t^2 & 3 t^4+5 t^3+3 t^2 & t+1 & 0 \\\hline
			8	 & t^7+4 t^6+2 t^5-2 t^4-t^3 & 3 t^5+7 t^4+4 t^3-t & 2 t^2+2 t+1 & 0 \\\hline
			9	 & t^8+4 t^7+3 t^6-3 t^5-2 t^4 & 4 t^6+9 t^5+7 t^4-2 t^2-t & 3 t^3+4 t^2+2 t+1 & 0 \\\hline
			10	 & t^9+5 t^8+4 t^7-3 t^6-3 t^5 & 4 t^7+12 t^6+10 t^5+t^4-3 t^3-2 t^2 & 4 t^4+6 t^3+4 t^2+t & 1 \\\hline
			11	 & t^{10}+5 t^9+5 t^8-4 t^7-5 t^6 & 5 t^8+15 t^7+15 t^6+2 t^5-5 t^4-4 t^3-t^2 & 5 t^5+10 t^4+7 t^3+3 t^2 & t+1 \\\hline
			12	 & t^{11}+6 t^{10}+7 t^9-3 t^8-6 t^7+t^5 & 5 t^9+18 t^8+19 t^7+4 t^6-8 t^5-7 t^4-2 t^3 & 7 t^6+14 t^5+12 t^4+5 t^3-t & 2 t^2+2 t+1 \\\hline
			13	 & t^{12}+6 t^{11}+8 t^{10}-4 t^9-9 t^8-t^7+t^6 & 6 t^{10}+22 t^9+27 t^8+7 t^7-10 t^6-11 t^5-4 t^4 & 8 t^7+20 t^6+18 t^5+9 t^4-2 t^2-t & 3 t^3+4 t^2+2 t+1 \\\hline
			14	 & t^{13}+7 t^{12}+10 t^{11}-3 t^{10}-11 t^9-2 t^8+2 t^7 & 6 t^{11}+26 t^{10}+34 t^9+12 t^8-13 t^7-16 t^6-6 t^5+t^3 & 10 t^8+26 t^7+27 t^6+13 t^5-5 t^3-3 t^2-t & 5 t^4+7 t^3+5 t^2+2 t+1 \\\hline
		\end{tabular}
	\end{table}
		From the above table, we observe that the E-polynomials of the strata $E\left( B^{[1]}_2 ;t\right)$, $E\left( B^{[3]}_3 ;t\right)$ and $E\left( B^{[6]}_4 ;t\right)$ are equal to $1$, which is the E-polynomial of a point. 
Indeed, each strata $B^{[1]}_2$, $B^{[3]}_3$ and $B^{[6]}_4$ contains a single monomial ideal:
\[\Yvcentermath1\Yboxdim{16pt}\begin{array}{ccc}
\gyoung(:y,~:x) & \gyoung(:y^2,~:<xy>,~~:<x^2>) &  \gyoung(:<y^3>,~:<y^2x>,~~:<x^2y>,~~~:<x^3>) \\
B^{[1]}_2=\{\maxideal\}& B^{[3]}_3=\{\lr x^2,xy,y^2\rr\} & B^{[6]}_4=\{\lr x^3,x^2y,y^2x,y^3\rr\}
\end{array}\]
	

\end{document}